\documentclass[11pt,reqno]{amsart}
\usepackage{amsmath,amssymb}
\usepackage[utf8]{inputenc}
\usepackage[paperwidth=8.26in,paperheight=11.69in,
left=1in, right=1in, bottom=1.2in]{geometry}

\usepackage{hyperref}
\usepackage{thmtools}
\declaretheoremstyle[
 headfont=\normalfont\bfseries,
 headindent= 0pt,
 bodyfont=\em,
 spaceabove=8pt,
 spacebelow=8pt
]{thm}
\declaretheoremstyle[
 headfont=\normalfont\em,
 headindent= 0pt,
 spaceabove=8pt,
 spacebelow=8pt
]{remark}
\declaretheoremstyle[
 headfont=\normalfont\bfseries,
 headindent= 0pt,
 spaceabove=8pt,
 spacebelow=8pt
]{example}
\declaretheoremstyle[
 headfont=\normalfont\bfseries,
 headindent= 0pt,
 spaceabove=8pt,
 spacebelow=8pt
]{definition}
\declaretheorem[name=Theorem,style=thm,numberwithin=section,
]{thm}
\declaretheorem[name=Proposition,style=thm,sibling=thm]{prop}
\declaretheorem[name=Lemma,style=thm,sibling=thm]{lem}
\declaretheorem[name=Corollary,style=thm,sibling=thm]{cor}
\declaretheorem[name=Example,style=example,sibling=thm]{example}

\declaretheorem[name=Remark,style=remark]{rem}

\usepackage{cleveref}
\crefname{thm}{Theorem}{Theorems}
\crefname{prop}{Proposition}{Propositions}
\crefname{lem}{Lemma}{Lemmas}
\crefname{cor}{Corollary}{Corollaries}
\crefname{example}{Example}{Examples}
\crefname{defn}{Definition}{Definitions}
\crefname{rem}{Remark}{Remarks}
\crefname{section}{Section}{Sections}

\crefname{equation}{}{}

\numberwithin{equation}{section}
\newcommand{\RR}{\mathbb{R}}
\newcommand{\CC}{\mathbb{C}}

\newcommand{\aba}{\bar{\alpha}}
\newcommand{\bba}{\bar{\beta}}

\newcommand{\gba}{\bar{\gamma}}
\newcommand{\nba}{\bar{\nu}}

\newcommand{\sba}{\bar{\sigma}}
\newcommand{\wba}{\bar{w}}
\newcommand{\Wba}{\overline{W}}
\newcommand{\zba}{\bar{z}}
\newcommand{\Zba}{\overline{Z}}
\newcommand{\mean}{H}

\newcommand{\jba}{\bar{j}}
\newcommand{\kba}{\bar{k}}
\newcommand{\lba}{\bar{l}}
\newcommand{\vr}{\varrho}

\DeclareMathOperator{\vol}{Vol}

\renewcommand{\Re}{\operatorname{Re}}
\renewcommand{\Im}{\operatorname{Im}}

\newcommand{\II}{I\!I}
\begin{document}
\title[On the Chern--Moser--Weyl tensor of real hypersurfaces]{On the Chern--Moser--Weyl tensor of real hypersurfaces}
\author{Michael Reiter}
\address{Fakultät für Mathematik, Universität Wien, Oskar-Morgenstern-Platz 1, 1090 Wien, Austria}
\email{m.reiter@univie.ac.at}
\author{Duong Ngoc Son}
\address{Fakultät für Mathematik, Universität Wien, Oskar-Morgenstern-Platz 1, 1090 Wien, Austria}
\email{son.duong@univie.ac.at}
\begin{abstract}
	We derive an explicit formula for the well-known Chern--Moser--Weyl tensor for nondegenerate real hypersurfaces in complex space in terms of their defining functions. The formula is considerably simplified when applying to ``pluriharmonic perturbations'' of the sphere or to a Fefferman approximate solution to the complex Monge-Amp\`ere equation. As an application, we show that the CR invariant one-form $X_{\alpha}$ constructed recently by Case and Gover is nontrivial on each real ellipsoid of revolution in $\mathbb{C}^3$, unless it is equivalent to the sphere. This resolves affirmatively a question posed by these two authors in 2017 regarding the\linebreak (non-) local CR invariance of the $\mathcal{I}'$-pseudohermitian invariant in dimension five and provides a counterexample to a recent conjecture by Hirachi.
\end{abstract}

\date{April 2, 2020}
\thanks{The first author was supported by the Austrian Science Fund FWF-project P28873-N35. The second author was supported by the Austrian Science Fund FWF-project M 2472-N35.}
\maketitle

\section{Introduction}\label{sec:intro}
The Chern--Moser--Weyl tensor $S_{\alpha\bba\gamma\sba}$, introduced in \cite{chern1974real}, is one of the most important biholomorphic invariants of nondegenerate real hypersurfaces in $\mathbb{C}^{n+1}$, $n\geq 2$. When $n=1$, it vanishes identically by default and its role is played by the Cartan invariant. A fundamental property of it is that $S_{\alpha\bba\gamma\sba}\equiv 0$ characterizes CR spherical hypersurfaces. Theses are hypersurfaces which are CR equivalent to the sphere or a real hyperquadric, see \cite{chern1974real}. Moreover, $S_{\alpha\bba\gamma\sba}$ {plays an important role in recent studies of higher CR invariants and ``secondary'' invariants, similar to the role of the Weyl tensor in conformal geometry;} see, e.g., \cite{case2013paneitz,hirachi2014q,hirachi2017variation,case2017p} and the references therein. There exist explicit formulas for $S_{\alpha\bba\gamma\sba}$ in the literature, see \cite{chern1974real,webster2000holomorphic, webster2002remark,foo2018explicit}. However, the formulas given in the aforementioned papers are difficult to compute in certain examples. For instance, although $S_{\alpha\bba\gamma\sba}$ is given by appropriate coefficients in a normal form \cite{chern1974real}, the normalization process is often too complicated; even for a hypersurface which is already given in normal form at a centered point, it is not practical to renormalize the hypersurface at near by points to compute the tensor. Due to this complexity, it is hard to apply them in certain situations, e.g. when locating the CR umbilics or studying the CR invariance of the $\mathcal{I}'$-curvature; see, e.g., \cite{webster2002remark,webster2000holomorphic,case2017p}.

This motivates the first goal of this paper. We provide an explicit formula for the Chern--Moser--Weyl tensor of nondegenerate real hypersurfaces in terms of arbitrary defining functions, which has a rather concise representation and allows for direct applications, as we demonstrate in this paper. In order to describe the formula, we need to introduce some notation. Let $M \subset \mathbb{C}^{n+1}$ be a real hypersurface and $\vr$ a (smooth) defining function for $M$, i.e., $M = \{\vr =0\}$ and $d\vr \ne 0$ on $M$. Let $(z,w) = (z_1, \dots , z_n,z_{n+1})$ for coordinates on $\mathbb{C}^{n+1}$, $\theta = \iota^{\ast}(i\bar{\partial}\vr)$ the pseudohermitian structure (in the sense of \cite{webster1978pseudo}) induced by $\vr$, for $\iota: M \rightarrow \CC^{n+1}$ is the inclusion, and $\nabla$ the associated Tanaka--Webster connection introduced in \cite{tanaka1975differential} and \cite{webster1978pseudo} (see \cite{dragomir--tomassini} for more details). Since $d\vr \ne 0$ on $M$, for local considerations we may assume, without loss of generality, that $\vr_{w}: = \partial \vr/\partial w \ne 0$. Under this condition, the vector fields of $(1,0)$-type $Z_{\alpha}: = \partial_{\alpha} - (\vr_{\alpha}/\vr_w)\, \partial_w$, $\alpha = 1,2,\dots, n$, form a basis of $T^{1,0}M$. In this paper, tensorial quantities will be expressed in this frame.

A dual coframe $\{\theta^{\alpha} \colon \alpha = 1,2,\dots , n\}$ to $\{Z_{\alpha}\}$ is given by 
\begin{equation} 
\theta^{\alpha} = dz^{\alpha} - i\xi^{\alpha} \theta,
\end{equation} 
where the $\xi^{k}$'s are the components of the $(1,0)$-complex vector field $\xi$ defined by
\begin{equation} 
	\xi \ \rfloor \ i\partial\bar{\partial} \vr = ir \bar{\partial} \vr ,
	\quad
	\partial\vr (\xi) = 1.
\end{equation} 
This coframe is admissible in the sense that $d\theta = i h_{\alpha\bba} \theta^{\alpha} \wedge \theta^{\bba}$ for some hermitian matrix $h_{\alpha\bba}$, which is called the \textit{Levi matrix}.

Various expressions in this paper can be written concisely by using the following second order differential operator (introduced earlier in \cite{li--luk}):
\begin{equation}
D_{\alpha\bba}^{\varrho}
:=
\partial_{\alpha}\partial_{\bba} - \frac{\varrho_{\alpha}}{\varrho_w}\partial_{w}\partial_{\bba}
- \frac{\varrho_{\bba}}{\varrho_{\bar w}}\partial_{\bar w}\partial_{\alpha} + \frac{\varrho_{\alpha}\varrho_{\bba}}{|\varrho_w|^2}\partial_w \partial_{\bar w}.
\end{equation}
Notice that $h_{\alpha \bba}$ in the frame $Z_{\alpha}$ is given by
\begin{equation}\label{e:levimt}
	h_{\alpha\bba}: = -id\theta(Z_{\alpha},Z_{\bba}) = \vr_{Z\bar{Z}} (Z_{\alpha},Z_{\bba}) = D_{\alpha\bba}^{\vr}(\vr),
\end{equation} 
where $\vr_{Z\bar{Z}}$ is the hermitian Hessian of $\vr$. Similarly, we define
\begin{equation}\label{e:dab}
D_{\alpha\beta}^{\varrho}
:=
\partial_{\alpha}\partial_{\beta} - \frac{\varrho_{\alpha}}{\varrho_w}\partial_{w}\partial_{\beta}
- \frac{\varrho_{\beta}}{\varrho_w}\partial_{w}\partial_{\alpha} + \frac{\varrho_{\alpha}\varrho_{\beta}}{\varrho_w^2}\partial_w^2,
\end{equation}
which satisfies
\begin{equation} 
	D_{\alpha\beta}^{\vr}(\varphi) = \varphi_{ZZ}(Z_{\alpha}, Z_{\beta}),
\end{equation} 
where $\varphi_{ZZ}$ is the Hessian of $\varphi$ in holomorphic coordinates. Since $M$ is nondegenerate, $h_{\alpha\bba}$ is invertible with inverse $h^{\bba \alpha}$ and we shall use these matrices to lower and raise the Greek indices, which run over $1,\ldots, n$. Throughout this article the summation convention is used and 
performed with respect to repeated indices.

In our first result, the defining function $\vr$ has nondegenerate complex Hessian, i.e., $\vr_{j\kba}$ is invertible. In this case, the inverse of the Levi matrix is given by (see, e.g., \cite[(2.7)]{li--lin--son}):
\begin{equation}
\label{e:inverseLevi}
h^{\bba \alpha} = \varrho^{\bba \alpha} - \frac{\varrho^{\bba} \varrho^\alpha}{|\partial \varrho|^2}.
\end{equation} 
Here $|\partial\varrho|^2:= |\partial \vr|^2_\omega$ is the squared norm of $\partial\vr $ in the Kähler metric $\omega: = i\partial\bar{\partial}\vr$ and $\vr^k = \vr^{k\lba}\vr_{\lba}$. We also use $\rho_{k\lba}$ and its inverse $\rho^{k\lba}$ to lower and raise the lowercase Latin indices, which run over $1,\ldots, n+1$. 

\begin{thm}\label{thm:chern-moser}
	Suppose that $M$ is defined by $\vr = 0$ with $\vr_{j\kba} = \delta_{jk}$ and $\theta: = i\bar{\partial}\vr$. Put $h_{\alpha\beta} = D_{\alpha\beta}^{\vr}(\vr)$, $h_{\bba\sba} = \overline{h_{\beta\sigma}}$, and $h_{\bba}^{\mu} = h_{\bba\sba}h^{\mu\sba}$. Then the pseudohermitian curvature and the Chern--Moser--Weyl tensor are given by
	\begin{align}
	R_{\alpha\bba\gamma\sba}
	= &
	|\partial\vr |^{-2} \left(h_{\alpha\bba}h_{\gamma\sba} + h_{\alpha\sba}h_{\gamma\bba} -h_{\alpha\gamma}h_{\bba \sba} \right), \label{e:phc}\\
	S_{\alpha\bba\gamma\sba}
	= &
	-\frac{h_{\alpha\gamma}h_{\bba \sba}}{|\partial\vr |^{2}} + \frac{h_{\mu\alpha}h^{\mu}_{\bba}h_{\gamma\sba} + h_{\mu\gamma}h^{\mu}_{\bba}h_{\alpha\sba} + h_{\mu\alpha}h^{\mu}_{\sba}h_{\gamma\bba} + h_{\mu\gamma}h^{\mu}_{\sba}h_{\alpha\bba}}{(n+2)|\partial\vr |^{2}} \notag 	\label{e:chern-moser} \\
	& - \frac{h_{\mu\nu}h^{\nu\mu} \left( h_{\alpha\bba}h_{\gamma\sba} + h_{\alpha\sba}h_{\gamma\bba}\right)}{(n+1)(n+2)|\partial\vr |^{2}},
	\end{align}
	in the local frame $Z_{\alpha}:= \partial_\alpha - (\vr_\alpha/\vr_w)\, \partial_w.$
\end{thm}

Note that in, e.g., \cite{ebenfelt2004rigidity},
a formula for the Chern--Moser--Weyl tensor similar to \eqref{e:chern-moser} was established in terms of coefficients of the second fundamental form of a CR immersion into the sphere which were not explicit.
 
The above statements follow from a more general result, given in \cref{thm:chern--moser-general}. In the general case, the formula is {inevitably} complicated. However, in the case of \textit{plurihamonic perturbations of the sphere}, i.e., when $\vr = \|Z\|^2 + \Re(\psi(Z))$ for some holomorphic function $\psi$, \cref{e:chern-moser} only involves 2nd order derivatives of the defining function although the Chern--Moser--Weyl tensor contains 4th order derivatives in general. 

Another important situation where our formulas are simplified is that of \textit{Fefferman approximate solution to the complex Monge-Amp\`ere equation}, i.e., when $J(\vr) = 1 + O(\vr^{n+2})$. Here, $J(\vr)$ is the \textit{Levi--Feffermann determinant} defined by
\begin{equation} \label{e:lfdetdef}
J(\vr):= -\det \begin{pmatrix}
	\vr & \vr_{\bar k} \\
	\vr_{j} & \vr_{j\bar k}
	\end{pmatrix}.
\end{equation} 
In this case, the formula for $S_{\alpha\bba\gamma\sba}$ is also considerably simplified; see \cref{cor:chern--moser}.

To derive our results we use the Gauß equation for ``semi-isometric'' immersions of pseudohermitian manifolds into K\"ahler manifolds. More precisely, we consider $(M,\theta: = -i\partial\vr$) as a pseudohermitian submanifold of the K\"ahler manifold $\CC^{n+1}$ equipped with the metric $\omega: = i\partial\bar{\partial}\vr$. Then, $d\theta = \iota^{\ast}\omega$ and hence $\iota$ is semi-isometric in the sense of \cite{son2019semi}. By the Gauß equation, the pseudohermitian curvature of $\theta$ is obtained from the K\"ahler curvature of $\omega$ and the second fundamental form. Using this fact, our computations become rather simple, since the second fundamental form $\II$ only involves derivatives of $\vr$ of order at most three. 

The second purpose of this paper is to give an affirmative answer to a question posed recently by Case and Gover. In \cite{case2017p}, Case and Gover constructed a pseudohermitian invariant $\mathcal{I}'$ in dimension five ($n=2$), namely,
\begin{equation}\label{e:defIprime}
\mathcal{I}'
=
-\frac{1}{8}\Delta_b |S_{\alpha\bba\gamma\sba}|^2 + \frac14\left|S_{\alpha\bba\gamma\sba,}{}^{\sba}\right|^2 + \frac{1}{12}R |S_{\alpha\bba\gamma\sba}|^2,
\end{equation} 
where $R$ is the Webster's scalar curvature. The formula for $\mathcal{I}'$ was stated in an equivalent form in \cite{case2017p} as for the middle term the CR analogue of the Cotton tensor $V_{\alpha\bba\gamma}$ was used (see \cite{gover2005sublaplace,case2017p}). They proved that the total $\mathcal{I}'$-curvature is a secondary invariant, at least in the case $c_2(H^{1,0}) = 0$, in the sense that 
\begin{equation}
	\int_{M^5} \widetilde{\mathcal{I}}' \widetilde{\theta} \wedge (d\widetilde{\theta})^2
	=
	\int_{M^5} \mathcal{I}' \theta \wedge (d{\theta})^2,
\end{equation}
for arbitrary pseudo-Einstein structures $\theta$ and $\widetilde{\theta}$. Moreover, as a local pseudohermitian invariant, $\mathcal{I}'$ transforms as follows: If $\widetilde{\theta} = e^{\Upsilon} \theta$, then by \cite[(8.17)]{case2017p}
\begin{equation}\label{e:2:8.17}
e^{3\Upsilon} \widetilde{\mathcal{I}}' = \mathcal{I}' + 2\Re X^{\gamma} \nabla_{\gamma} \Upsilon,
\end{equation} 
where
\begin{equation}\label{e:Xalphadef}
X_{\alpha} = \frac{1}{2} S_{\alpha\bba\gamma\sba}S^{\bba\gamma\sba}{}_{\bar{\epsilon},}{}^{\bar{\epsilon}} + \frac14 \nabla_{\alpha}|S_{\epsilon\bba\gamma\sba}|^2.
\end{equation}
As discussed in \cite[Remark~8.11]{case2017p}, $\mathcal{I}'$ can be formally regarded as the ``prime analogue'' of the conformal invariant $ \left|\widetilde{\nabla}\widetilde{\mathrm{Rm}}\right|^2$ of Fefferman and Graham, where
$\widetilde{\nabla}$ and $\widetilde{\mathrm{Rm}}$ are the covariant derivative and Riemannian curvature tensor, respectively, of the ambient metric; see \cite[(9.3)]{fefferman2012ambient}.
A question posed by Case and Gover in \cite{case2017p} asks whether there exists a $5$-dimensional pseudo-Einstein manifold for which $X_{\alpha}$ is nonzero. By using \cref{thm:chern-moser}, we show that in fact $X_{\alpha}$ is a nontrivial CR invariant for generic real ellipsoidal hypersurfaces of revolution in $\mathbb{C}^3$, which appeared in \cite{webster2002remark}, and consequently, $\mathcal{I}'$ is not a local CR invariant.

\begin{thm}\label{thm:xalpha} The CR invariant one-form $X_{\alpha}$ is nontrivial on real ellipsoids of revolution $E(a)$ in $\CC^3$ defined by 
	\begin{equation}
	\varrho (z_1,z_2,w): = |z_1|^2 + |z_2|^2 + |w|^2 + \Re (aw^2) - 1 = 0, \quad a\in \RR,
	\end{equation}
unless $a=0$.
\end{thm}
In fact, we shall give an explicit formula for the CR invariant one-form $X_{\alpha}$ on $E(a)$ which is manifestly nontrivial.

As {briefly explained} in \cite[Remark~8.12]{case2017p}, a pseudo-Einstein CR manifold for which $X_{\alpha} \not\equiv 0$ provides a counterexample to a conjecture by Hirachi regarding the decomposition of the scalar secondary invariants on compact CR manifolds. Precisely, in \cite[p. 242]{hirachi2014q}, it is conjectured that a pseudohermitian scalar invariant for which the integral is a secondary invariant can be decomposed into the sum of a constant multiple of the $Q'$-curvature, a local CR invariant, and a divergence.
{As \cite{case2017p} does not contain full details, we sketch an argument suggested to the authors by the referee that disproves the Hirachi conjecture as follows. It can be shown that in the situation of \cref{thm:xalpha} the divergence $\Re \nabla^{\alpha} X_{\alpha}$ is not identically zero on $E(a)$, see Remark~1 at the end of \cref{sec:Xalpha}. Then there exists a smooth function $\Upsilon$ on ${E(a)}$ such that 
\[
	\Re \int_{E(a)}   X^{\alpha} \nabla_{\alpha} \Upsilon \, \theta \wedge (d\theta)^2 = -  \int_{E(a)} \Upsilon \Re(\nabla^{\alpha} X_{\alpha}) \, \theta \wedge (d\theta)^2 \ne 0.	
\]
Therefore, using \cref{e:2:8.17} above, we obtain
\[
	\frac{d}{dt}\biggl|_{t=0}\, \int_{E(a)} \mathcal{I}'_{e^{t\Upsilon} \theta} e^{3t\Upsilon}\, \theta \wedge (d\theta)^2  = 2\Re \int_{E(a)} X^{\alpha} \nabla_{\alpha} \Upsilon \, \theta \wedge (d\theta)^2 \ne 0.
\]
Thus, the total $\mathcal{I}'$ is not CR invariant and hence $\mathcal{I}'$ cannot be the sum of a local CR invariant and a pure divergence. 
}

We note that the one-form $X_{\alpha}$ vanishes identically on CR spherical manifolds. More generally, it vanishes identically on CR manifolds {for which there exists a pseudohermitian structure with} parallel Chern--Moser--Weyl tensor, i.e., when $\nabla S_{\alpha\bba\gamma\sba} = 0$. Thus, it is still an interesting open question whether Hirachi's conjecture is true on CR spherical manifolds. It is worth pointing out that there exist examples showing that  the CR sphericity of the manifold is not necessary for the vanishing of $X_\alpha$; see \cref{ex:parallelchernmoser}.

The paper is organized as follows. In \cref{sec:sff}, we study the second fundamental form of real hypersurfaces that are semi-isometrically immersed in a Kähler manifold. The result in this section is crucial for the next section. In \cref{sec:cmw}, we give explicit formulas for the pseudohermitian curvature tensor and the Chern--Moser--Weyl tensor for general real hypersurfaces and prove \cref{thm:chern-moser}. In \cref{sec:Xalpha}, we compute the one-form $X_\alpha$ on the real ellipsoids of revolution in $\mathbb{C}^3$ and prove \cref{thm:xalpha}. In the last section, we give an example of a family of locally equivalent nonspherical CR manifolds with parallel Chern--Moser--Weyl tensor and, as a simple application of our formula \cref{e:chern-moser}, show that the hypersurfaces in this family are pairwise inequivalent globally.
\section{Real hypersurfaces in K\"ahler manifolds and second fundamental form}\label{sec:sff}
As briefly explained in the introduction, our approach to the Chern--Moser--Weyl tensor is via the 
Gauß equation, derived recently in \cite{son2019semi}. For this approach, we shall compute explicitly 
the second fundamental form of a real hypersurface in $\mathbb{C}^{n+1}$, viewed as a CR submanifold of
a Kähler manifold with an appropriate metric.

Let $M\subset \mathbb{C}^{n+1}$ be a nondegenerate real hypersurface defined by $\varrho = 0$ with $d\varrho \ne 0$ on~$M$. It is well-known (see, e.g., \cite{li--son,farris1986intrinsic}) that there is a vector field $\xi$ of type $(1,0)$ such that
\begin{equation}
	\xi \ \rfloor \ i\partial \bar{\partial} \vr = ir \bar{\partial}{\vr},
	\quad
	\partial\vr (\xi) = 1.
\end{equation}
The function $r$, given by $r = \rho_{j\kba} \xi^j \xi^{\bar{k}}$, is often called the \textit{transverse curvature} of the defining function.

We first suppose that the complex Hessian $\varrho _{j\kba}$ is nondegenerate so that $\varrho $ is a K\"ahler potential for a (pseudo-) Kähler metric $\omega$ on a neighborhood $U$ of~$M$ in $\CC^{n+1}$. In this situation, it can be shown that $r = |\xi|^2_{\omega} = |\partial\vr|^{-2}_{\omega}$, for $\omega$ being the K\"ahler metric with potential $\vr$, i.e., $\omega = i\partial\bar{\partial}\vr$. Moreover, $\iota \colon (M,\theta) \to (U,\omega)$ is a semi-isometric CR immersion in the sense of \cite{son2019semi}, i.e. $\iota^* \omega = d \theta$. 

Let $\nabla$ and $\widetilde{\nabla}$ be the Tanaka--Webster connection of $(M , \theta)$ and the Chern connection of $(U,\omega)$, {respectively}. Then the second fundamental form of $M$ is defined by the Gauß formula (see \cite{son2019semi})
\begin{equation} 
	\II(Z,W) := \widetilde{\nabla}_{\widetilde{Z}}\widetilde{W} - \nabla_ZW.
\end{equation} 
Here $\widetilde{Z}$ and $\widetilde{W}$ are smooth extensions of $Z$ and $W$ to a neighborhood of $M$ in $U$.

Taking the trace of $\II$ on {horizontal directions}, we obtain the $(1,0)$-mean curvature vector field $H$. Namely,
\begin{equation}
	H: = \frac1n\sum_{\alpha =1}^{n} \II(Z_{\aba} , Z_{\alpha}).
\end{equation}
Basic properties of $\II$ have been studied in \cite{son2019semi}. In particular, Gauß--Codazzi--Mainardi equations relating the Tanaka--Webster curvature and the torsion to the curvature of $\omega$ have been proved. In the following, the convention for the curvature operator of $\nabla$ is
\begin{equation}
	R(X,Y)Z = \nabla_X\nabla_YZ - \nabla_Y \nabla_X Z - \nabla_{[X,Y]}Z.
\end{equation}
The torsion $\mathbb{T}_{\nabla}$ of the Tanaka--Webster connection is nontrivial:
\begin{equation}
	\mathbb{T}_{\nabla} (X,Y) = \nabla_XY - \nabla_{Y}X - [X,Y].
\end{equation}
If $T$ is the characteristic direction associated to $\theta$, i.e., $T$ is the unique real vector field on $M$ that satisfies
\begin{equation}
	T \ \rfloor \ d\theta = 0,
	\quad 
	\theta(T) = 1,
\end{equation}
then the pseudohermitian torsion is defined by
\begin{equation}
	\tau X := \mathbb{T}_{\nabla} (T,X).
\end{equation}
The curvature of the Chern connection of $\omega$ will be denoted by $\widetilde{R}$. The aforementioned Gauß equations are given as follows:
\begin{prop}[Gauß equations \cite{son2019semi}]\label{prop:ge}
	Let $\iota \colon (M, \theta) \hookrightarrow (\mathcal{X},\omega)$ be a pseudohermitian CR submanifold of a Kähler manifold. Let $R$ and $\widetilde{R}$ be the curvature operators of the Tanaka--Webster and Chern connection on $M$ and $\mathcal{X}$, respectively. Then
	\begin{enumerate}
		\item for $X,Z \in \Gamma(T^{1,0}M)$ and $\overline{Y},\overline{W} \in \Gamma(T^{0,1} M) $, the following Gauß equation holds:
		\begin{align}\label{e:gauss}
		\langle \widetilde{R}(X,\overline{Y}) Z, \Wba\rangle
		& =
		\langle R(X,\overline{Y}) Z, \Wba\rangle
		+
		\langle \II (X,Z) , \II (\overline{Y}, \Wba) \rangle \notag \\
		& \qquad - |\mean |^2 \left(\langle \overline{Y} , Z \rangle \langle X ,\Wba \rangle + \langle X , \overline{Y} \rangle \langle Z , \Wba \rangle \right),
		\end{align}
		\item for $X,Z \in \Gamma(T^{1,0}M)$,
		\begin{equation}\label{e:gausstorsion}
		\langle \tau X , Z \rangle 
		=
		-i \langle \II(X,Z) , \overline{\mean} \rangle.
		\end{equation}
	\end{enumerate}
\end{prop}
We point out that although the proof given in \cite{son2019semi} is for the strictly pseudoconvex case, it also works for the Levi-nondegenerate case.
 
In order to apply these equations, we need to compute the second fundamental form $\II$ in terms of the defining function. In fact,
it was proved in \cite{son2019semi} that, using our notation, 
\begin{equation}
	\II(Z_{\aba} , Z_{\beta})
	= 
	-h_{\beta\aba} \xi,
\end{equation}
which implies $H = -\xi$ and $r = |H|^2$. 

Below, we shall compute the ``holomorphic" part $\II(Z_{\alpha} , Z_{\beta})$ of the second fundamental form. For this purpose, we need the following formula for the Tanaka--Webster connection forms $\omega_{\beta}{}^{\gamma}$ computed by Li and Luk \cite{li--luk} (see also \cite{webster1978pseudo}). Recall that the connection forms are defined by $\nabla Z_{\beta} = \omega_{\beta}{}^{\gamma} \otimes Z_{\gamma}$, and 
\begin{equation}\label{e:cf}
	\omega_{\beta}{}^{\gamma}
	=
	\left(h^{\gamma\sba} Z_{\mu} h_{\beta\sba} - \xi_{\beta}\delta_{\mu}^{\gamma}\right) \theta^{\mu} + \xi^{\gamma} h_{\beta\bar{\mu}} \theta^{\bar{\mu}} - iZ_{\beta} \xi^{\gamma} \theta,
\end{equation}
where $\xi_{\beta} = h_{\beta\sba} \xi^{\sba}$, see \cite[(2.19)]{li--luk}.

Then we have the following formula for $\II(Z_{\alpha} , Z_{\beta})$:

\begin{prop}
	Let $M$ be a nondegenerate real hypersurface in $\mathbb{C}^{n+1}$ defined by $\varrho = 0$ with $d\varrho\ne 0$. Assume that $\varrho$ has nondegenerate complex Hessian. Put $\omega = i\partial\bar{\partial} \vr$ and $\theta = \iota^{\ast}(i\bar{\partial}\vr)$. Then the inclusion $\iota \colon (M,\theta) \to (U,\omega)$ is a semi-isometric immersion. Moreover, if $\varrho _w \ne 0$ and $Z_{\alpha} = \partial_{\alpha} - (\varrho _{\alpha}/\varrho _w) \partial_w$, then
\begin{equation}\label{e:sff}
	\II(Z_{\alpha} , Z_{\beta})
	=
	\left(\varrho^{\kba} D^{\varrho}_{\alpha\beta}(\varrho_{\kba}) - h_{\alpha\beta}\right) \xi,
\end{equation}
	where $D^{\vr}_{\alpha\beta}$ is the 2nd order differential operator defined in \cref{e:dab} and $h_{\alpha\beta} = D^{\varrho}_{\alpha\beta}(\varrho)$.
\end{prop}
\begin{proof}
	 Let $\widetilde{\nabla}$ be the Chern connection of $\omega:=i\partial\bar{\partial}\varrho $ and let $\Gamma^k_{jl}$ be its Christoffel symbols given by $\Gamma^k_{jl} = \varrho^{k\bar{m}}\partial_{j}\varrho_{l\bar{m}}$. Put
	\begin{equation}\label{e:2.13}
	U^k_{\alpha\beta}
	=
	\varrho^{k\bar{l}} D_{\alpha\beta}^{\varrho}(\varrho_{\bar{l}})
	=
	\Gamma^k_{\alpha\beta} - \frac{\varrho _{\alpha}}{\varrho _w} \Gamma^{k}_{w\beta} 
	-
	\frac{\varrho _{\beta}}{\varrho _w} \Gamma^{k}_{w\alpha}
	+
	\frac{\varrho _{\alpha}\varrho _{\beta}}{\varrho _w^2} \Gamma^{k}_{ww}.
	\end{equation}
	Since $Z_{\alpha} = \partial_{\alpha} - (\vr_{\alpha}/\vr_w)\partial_w$, we have, after some simplification,
	\begin{equation}
	\widetilde{\nabla}_{Z_{\alpha}}Z_{\beta}
	=
	U^k_{\alpha\beta} \partial_k - (h_{\alpha\beta}/\varrho _w)\partial_w.
	\end{equation}
	From \cref{e:cf}, we obtain
	\begin{equation}
	{\nabla}_{Z_{\alpha}}Z_{\beta}
	=
	(h^{\gamma\sba}Z_{\alpha} h_{\beta\sba} - \xi_{\beta} \delta _{\alpha}^{\gamma})\partial_{\gamma} + (1/\varrho _w)( \varrho _{\alpha}\xi_{\beta} - \varrho _{\gamma}h^{\gamma\sba} Z_{\alpha}h_{\beta\sba})\partial_w.
	\end{equation}
 We obtain,
	\begin{align}\label{e:sfftrans2}
	\II(Z_{\alpha} , Z_{\beta})
	 = & ~
	\widetilde{\nabla}_{Z_{\alpha}}Z_{\beta} - {\nabla}_{Z_{\alpha}}Z_{\beta} \notag \\
	= & ~ 
	\left(U^{\gamma}_{\alpha\beta} - h^{\gamma\bar{\mu}} Z_{\alpha} (h_{\beta\bar{\mu}}) + \xi_{\beta} \delta^{\gamma}_{\alpha}
	\right) \partial_{\gamma} \notag \\
	& + (1/\varrho _w)\left(\varrho _{\gamma} h^{\gamma\bar{\mu}} Z_{\alpha} (h_{\beta\bar{\mu}}) - \varrho _{\alpha}\xi_{\beta} - h_{\alpha\beta} + \varrho _w U^w_{\alpha\beta}\right) \partial_w.
	\end{align}
	To simplify \cref{e:sfftrans2}, we compute directly that
	\begin{equation}
		Z_{\alpha}(\varrho_{\beta}/\varrho_w)
		=
		\frac{h_{\alpha\beta}}{\varrho_w},
		\quad 
		Z_{\alpha}(\varrho_{\bba}/\varrho_{\wba})
		=
		\frac{h_{\alpha\bba}}{\varrho_{\wba}},
	\end{equation}
	hence
	\begin{align}\label{e:a}
		Z_{\alpha}(h_{\beta\bar{\mu}})
		=
		\varrho_{\beta\bar{\mu}\alpha}- \frac{\varrho_{\alpha}\varrho_{\beta\bar{\mu} w}}{\varrho_w} - \frac{h_{\alpha\beta}\varrho_{\bar{\mu}w}}{\varrho_{w}} 
		- \frac{\varrho_{\beta}\varrho_{\bar{\mu}w\alpha}}{\varrho_w}
		+
		\frac{\varrho_{\alpha}\varrho_{\beta}\varrho_{\bar{\mu}ww}}{\varrho_w^2}
		-
		\frac{h_{\alpha\bar{\mu}}\varrho_{\beta \wba}}{\varrho_{\wba}} \notag \\
		- \frac{\varrho_{\bar{\mu}}\varrho_{\beta \wba \alpha}}{\varrho_{\wba}}
		+
		\frac{\varrho_{\alpha}\varrho_{\bar{\mu}}\varrho_{\beta\wba w}}{|\varrho_w|^2}
		+
		\frac{h_{\alpha\beta}\varrho_{\bar{\mu}}\varrho_{\wba w}}{|\varrho_w|^2}
		+
		\frac{h_{\alpha\bar{\mu}}\varrho_{\beta}\varrho_{\wba w}}{|\varrho_w|^2}
		+
		\frac{\varrho_{\bar{\mu}}\varrho_{\beta}\varrho_{\wba w \alpha}}{|\varrho_w|^2}
		+
		\frac{\varrho_{\alpha}\varrho_{\bar{\mu}}\varrho_{\beta}\varrho_{\wba w w}}{|\varrho_w|^2\varrho_w}.
	\end{align}
	Multiplying \cref{e:a} with $h^{\gamma\bar{\mu}}$,	we obtain, after simplification, that
	\begin{equation}\label{e:b}
		h^{\gamma\bar{\mu}} Z_{\alpha} (h_{\beta\bar{\mu}}) = 
		U^{\gamma}_{\alpha\beta} - \frac{\varrho^{\gamma}}{|\partial\varrho|^2}\varrho^{\kba}D^{\varrho}_{\alpha\beta}(\varrho_{\kba})
		+
		\frac{\varrho^{\gamma}h_{\alpha\beta}}{|\partial\varrho|^2}
		+\left(\frac{\varrho_{\beta}\varrho_{w\wba}}{|\varrho_w|^2}
		-\frac{\varrho_{\beta\wba}}{\varrho_{\wba}}\right) \delta_{\alpha}^{\gamma}.
	\end{equation}
	Plugging $\xi^{\gamma} = \varrho^{\gamma}/|\partial\varrho|^2$ into \cref{e:b}, we find that 
	\begin{equation}\label{e:c}
		 U^{\gamma}_{\alpha\beta} - h^{\gamma\bar{\mu}} Z_{\alpha} (h_{\beta\bar{\mu}}) + \xi_{\beta} \delta^{\gamma}_{\alpha}
		 =
		 \left(\varrho^{\kba}D^{\varrho}_{\alpha\beta}(\varrho_{\kba}) - h_{\alpha\beta}\right)\xi^{\gamma}.
	\end{equation}
	Similarly, using \cref{e:b} and \cref{e:2.13}, we obtain that
	\begin{equation}\label{e:d}
		\varrho _{\gamma} h^{\gamma\bar{\mu}} Z_{\alpha} (h_{\beta\bar{\mu}}) - \varrho _{\alpha}\xi_{\beta} - h_{\alpha\beta} + \varrho _w U^w_{\alpha\beta}
		=
		\frac{\vr_w \vr^w}{|\partial\vr|^2} \left(\varrho^{\kba}D^{\varrho}_{\alpha\beta}(\varrho_{\kba}) - h_{\alpha\beta}\right)
	\end{equation}
	Plugging \cref{e:c,e:d} into \cref{e:sfftrans2}, we find that
	\begin{equation} 
		\II(Z_{\alpha} , Z_{\beta})
		= \left(\varrho^{\kba}D^{\varrho}_{\alpha\beta}(\varrho_{\kba}) - h_{\alpha\beta}\right) \xi^j \partial_j,
	\end{equation} 
	which finishes the proof.
\end{proof}	
In a local frame $Z_{\alpha}$, the torsion tensor $\tau$ has components denoted by $A_{\alpha\beta}$, i.e., 
\begin{equation}
	\tau Z_{\alpha} = A_{\alpha}{}^{\bba} Z_{\bba}.
\end{equation}
We obtain the following formula for the torsion tensor which may be of independent interest.

\begin{cor}\label{cor:torsion}
Let $M$ be defined by $\vr = 0$ with $d\vr \ne 0$ and $\theta = i\bar{\partial}\vr$. Suppose that $\vr_w \neq 0$, then the torsion tensor $A_{\alpha\beta}$ in the local frame $Z_{\alpha}:= \partial_\alpha - (\vr_\alpha/\vr_w)\, \partial_w$ is given by:
\begin{equation}\label{e:torsion}
	-i A_{\alpha\beta}
	=
	\xi^{\kba} D^{\vr}_{\alpha\beta}(\vr_{\kba}) - |\xi|^2 h_{\alpha\beta}.
\end{equation}
\end{cor}
An alternative formula for the torsion was given in \cite[Theorem~1.1]{li2006explicit}. In fact, it was proved that, for strictly plurisubharmonic $\vr$, 
\begin{equation}\label{e:li--luk--torison}
	A_{\alpha\beta} = -\frac{i}{|\partial\vr|^2} Z_{\alpha}(\vr_{\kba}) Z_{\beta}(\vr^{\kba}).
\end{equation} 
One can check that \cref{e:torsion} and \cref{e:li--luk--torison} are equivalent when $\vr$ is strictly plurisubharmonic.
\begin{proof}[Proof of \cref{cor:torsion}]
	We first assume that $\vr_{j\kba}$ is invertible. Since $A_{\alpha\beta} = \langle \tau Z_{\alpha} , Z_{\beta}\rangle$, it follows from the Gauß equation \cref{e:gausstorsion} that
	\begin{align}
		A_{\alpha\beta} 
		& = -i \langle \II(Z_{\alpha},Z_{\beta}) , \bar{H} \rangle \notag \\
		& =
		i \II_{\alpha\beta} |H|^2 \notag \\
		& =
		\frac{i}{|\partial\vr|^2} \left(\vr^{\bar{k}} D^{\vr}_{\alpha\beta}(\vr_{\bar{k}}) - h_{\alpha\beta}\right) \notag \\
		& = i\left(\xi^{\kba} D^{\vr}_{\alpha\beta}(\vr_{\kba}) - |\xi|^2 h_{\alpha\beta}\right).
	\end{align}
	Here we have used the fact that $H = -\xi$ and $|H|^2 = r = |\partial\vr|^{-2}$. Thus, \cref{e:torsion} is proved in the case when $\vr_{j\kba}$ is nondegenerate.
	
	To remove the assumption that $\vr_{j\kba}$ is invertible, we use an idea taken from \cite{li--luk}. Precisely, we can replace $\vr$ by $\widetilde{\vr}: = \vr + C \vr^2$, for $C>0$ large enough, so that $\widetilde{\vr}_{j\kba}$ is non-degenerate. Observe that $\theta = -i\partial\vr = -i\partial\widetilde{\vr}$. To conclude the proof, we need to verify that the right-hand side of \cref{e:torsion} does not change when $\vr$ is replaced by $\widetilde{\vr}$. Indeed, we can verify directly (or alternatively use \cite[Lemma~7.1]{son2019semi}) that the following holds on~$M$:
	\begin{equation} 
		\widetilde{\xi}^{\kba} = \xi^{\kba},
		\quad
		|\widetilde{\xi}|_{\widetilde \omega}^2 = |\xi|_\omega^2 + 2C.
		\quad
	\end{equation} 
	Moreover, on $M$, $D^{\widetilde{\vr}}_{\alpha\beta} = D^{\vr}_{\alpha\beta}$ and thus
	\begin{equation} 
		D^{\widetilde{\vr}}_{\alpha\beta}(\widetilde{\vr}_{\kba})
		=D^{\vr}_{\alpha\beta}(\vr_k) + 2C \vr_{\kba} h_{\alpha\beta}.
	\end{equation} 
	Consequently, the right-hand side of \cref{e:torsion} does not change when $\vr$ is replaced by $\widetilde{\vr}$, since $\vr_{\bar k} \xi^{\bar k} = 1$, which completes the proof.
\end{proof}

\section{The Chern--Moser--Weyl tensor}\label{sec:cmw}
Similarly to the definitions of $D^{\vr}_{\alpha\beta}$ and $D^{\vr}_{\alpha\bba}$, we define a 4th order linear differential operator $\mathcal{R}^{\vr}_{\alpha\bba\gamma\sba}$ by the following equation:
\begin{equation} 
\mathcal{R}^{\vr}_{\alpha\bba\gamma\sba}(\varphi)
=
\varphi_{Z\Zba Z\Zba}(Z_{\alpha},Z_{\bba}, Z_{\gamma}, Z_{\sba}).
\end{equation} 
In particular, 
\begin{align} 
\mathcal{R}^{\vr}_{\alpha\bba\gamma\sba}(\vr) 
= &~
\vr_{\alpha\bba\gamma\sba} 
- \frac{\vr_{\alpha}\vr_{w\bba\gamma\sba}}{\vr_w}
- \frac{\vr_{\bba}\vr_{\alpha\wba\gamma\sba}}{\vr_{\wba}} 
- \frac{\vr_{\gamma}\vr_{\alpha\bba w\sba}}{\vr_{w}}
- \frac{\vr_{\sba}\vr_{\alpha\bba\gamma\wba}}{\vr_{\wba}}
 + \frac{\vr_{\alpha} \vr_{\bba}\vr_{w\wba\gamma\sba}}{|\vr_w|^2} 
 + \frac{\vr_{\alpha} \vr_{\gamma}\vr_{w\bba w\sba}}{\vr_w^2} \notag \\
& ~ + \frac{\vr_{\alpha} \vr_{\sba}\vr_{w\bba \gamma\wba}}{|\vr_w|^2}
+ \frac{\vr_{\bba} \vr_{\gamma}\vr_{\alpha \wba w\sba}}{|\vr_w|^2}
+ \frac{\vr_{\bba} \vr_{\sba}\vr_{\alpha \wba \gamma \wba}}{\vr_{\wba}^2}
+ \frac{\vr_{\gamma} \vr_{\sba}\vr_{\alpha\bba w\wba}}{|\vr_w|^2} 
- \frac{\vr_{\alpha} \vr_{\bba} \vr_{\gamma} \vr_{w\wba w\sba}}{|\vr_w|^2 \vr_w} \notag \\
& ~ - \frac{\vr_{\alpha} \vr_{\bba} \vr_{\sba} \vr_{w\wba \gamma \wba}}{|\vr_w|^2 \vr_{\wba}}
- \frac{\vr_{\bba} \vr_{\gamma} \vr_{\sba} \vr_{\alpha\wba w\wba}}{|\vr_w|^2 \vr_{\wba}} 
- \frac{\vr_{\alpha} \vr_{\gamma} \vr_{\sba} \vr_{w\bba w \wba}}{|\vr_w|^2 \vr_{w}}
+ \frac{\vr_{\alpha}\vr_{\bba}\vr_{\gamma} \vr_{\sba}\vr_{w\wba w\wba}}{|\vr_w|^4}.
\end{align} 
To work with general defining functions, we need to introduce some notation. Let $\psi_{j\kba} = \vr_{j\kba} + (1-r)\vr_j\vr_{\kba}$, then $\det(\psi_{j\kba}) = J(\vr)$ (see \cite{li--son}), and hence $\psi_{j\kba}$ is invertible. Let $\psi^{\kba j}$ be the inverse of $\psi_{j\kba}$ and
\begin{equation} 
	h^{j\kba} = \psi^{j\kba} - \xi^j\xi^{\kba}.
\end{equation} 
Then $h^{\bba\alpha}$ is the inverse of $h_{\alpha\bba}$, which can be verified by a direct computation, and when $|\partial\vr|^2 \ne 0$,
\begin{equation} 
	h^{j\kba} = \vr^{j\kba} - |\partial\vr|^2 \xi^{j}\xi^{\bar k}.
\end{equation} 

Our main result in this section is the following
\begin{thm}\label{thm:chern--moser-general}
	Let $M$ be a nondegenerate real hypersurface in $\mathbb{C}^{n+1}$ defined by $\vr = 0$ with $d\vr \ne 0$ and $J(\vr) \ne 0$ on $M$. Assume $\vr_w \neq 0$. Then the pseudohermitian curvature is given by
	\begin{align}\label{e:pc}
	R_{\alpha\bba\gamma\sba}
	& =
	- \mathcal{R}^{\vr}_{\alpha\bba\gamma\sba}(\vr)
	+ h^{j\kba} D^{\vr}_{\alpha\gamma}(\vr_{\kba}) D^{\vr}_{\bba\sba}(\vr_{j})
	+
	|\xi|^2 \left(h_{\alpha\bba} h _{\gamma\sba} + h_{\alpha\sba} h _{\gamma\bba}\right) \notag \\
	& \quad 
	+
	h_{\bba\sba}\xi^{\kba} D^{\vr}_{\alpha\gamma}(\vr_{\kba}) + h_{\alpha\gamma} \xi^{j} D^{\vr}_{\bba\sba}(\vr_{j}) - |\xi|^2 h_{\alpha\gamma} h_{\bba\sba}.
	\end{align}
	and the Chern--Moser--Weyl tensor is given by
	\begin{align}\label{e:cm}
	S_{\alpha\bba\gamma\sba} & =
	-\mathcal{R}^{\vr}_{\alpha\bba\gamma\sba}(\vr) + h^{j\kba} D^{\vr}_{\alpha\gamma}(\vr_{\kba}) D^{\vr}_{\bba\sba}(\vr_{j}) +
	h_{\bba\sba}\xi^{\kba} D^{\vr}_{\alpha\gamma}(\vr_{\kba}) + h_{\alpha\gamma} \xi^{j} D^{\vr}_{\bba\sba}(\vr_{j}) - |\xi|^2 h_{\alpha\gamma} h_{\bba\sba} \notag \\
	& \quad + \frac{1}{n+2}\left(h_{\gamma\sba} D^{\vr}_{\alpha\bba} + h_{\gamma\bba} D^{\vr}_{\alpha\sba} + h_{\alpha\bba} D^{\vr}_{\gamma\sba} + h_{\alpha\sba} D^{\vr}_{\gamma\bba}\right) \log J(\vr) \notag \\
	& \quad - \frac{1}{(n+1)(n+2)}\left(h_{\alpha\bba} h _{\gamma\sba} + h_{\alpha\sba} h _{\gamma\bba}\right) h^{\epsilon\bar{\delta}}D^{\vr}_{\epsilon\bar{\delta}}\log J(\vr),
	\end{align} 
in the local frame $Z_{\alpha}:= \partial_\alpha - (\vr_\alpha/\vr_w)\, \partial_w$.
\end{thm}
\begin{proof}
We first assume that $\vr_{j\kba}$ is invertible. Recall that the curvature of the K\"ahler metric $i\partial\bar{\partial}\vr$ is given by $\widetilde{R}_{\jba k \lba m} = -\varrho_{m \lba k \jba} + \varrho^{\bar{p} q}\vr_{mk\bar{p}} \vr_{\lba \jba q}$ in the coordinates $z_j, j = 1,2,\dots , n+1$ (see \cite[Proposition~6.2]{morrow2006complex}, but mind our opposite sign convention for the curvature operator on K\"ahler manifolds). 
Then
\begin{align}
	\widetilde{R}(Z_{\alpha} , Z_{\bba} , Z_{\gamma} , Z_{\sba})
	=
	-\mathcal{R}^{\vr}_{\alpha\bba\gamma\sba}(\vr) + \vr^{j\kba} D^{\vr}_{\alpha\gamma}(\vr_{\kba}) D^{\vr}_{\bba\sba}(\vr_{j}).
\end{align}
Plugging this into the Gauß equation \cref{e:gauss}, we have that
\begin{align} 
R_{\alpha\bba\gamma\sba}
& =
- \mathcal{R}^{\vr}_{\alpha\bba\gamma\sba}(\vr) + \vr^{j\kba} D^{\vr}_{\alpha\gamma}(\vr_{\kba}) D^{\vr}_{\bba\sba}(\vr_{j})
+
|\xi|^2 \left(h_{\alpha\bba} h _{\gamma\sba} + h_{\alpha\sba} h _{\gamma\bba}\right) \notag \\
& \quad - |\xi|^2\left(\vr^{\bar{p}} D^{\vr}_{\alpha\gamma}(\vr_{\bar{p}}) - h_{\alpha\gamma}\right) \left(\vr^{q} D^{\vr}_{\bba\sba}(\vr_{q}) - h_{\bba\sba}\right) \notag \\
& =
- \mathcal{R}^{\vr}_{\alpha\bba\gamma\sba}(\vr)
+ h^{j\kba} D^{\vr}_{\alpha\gamma}(\vr_{\kba}) D^{\vr}_{\bba\sba}(\vr_{j})
+
|\xi|^2 \left(h_{\alpha\bba} h _{\gamma\sba} + h_{\alpha\sba} h _{\gamma\bba}\right) \notag \\
& \quad 
+
h_{\bba\sba}\xi^{\kba} D^{\vr}_{\alpha\gamma}(\vr_{\kba}) + h_{\alpha\gamma} \xi^{j} D^{\vr}_{\bba\sba}(\vr_{j}) - |\xi|^2 h_{\alpha\gamma} h_{\bba\sba},
\end{align}
which shows \cref{e:pc}, as desired.

The Chern--Moser--Weyl tensor can be obtained by taking the complete tracefree part of $R_{\alpha\bba\gamma\sba}$. To this end, we use the formula for the Ricci tensor obtained by Li--Luk \cite{li--luk}, namely,
\begin{equation}
\label{e:lilukRicci}
R_{\alpha\bba} = - D^{\vr}_{\alpha\bba} \log J(\vr) + (n+1)|\xi|^2 h_{\alpha\bba}.
\end{equation}
This and \cref{e:pc} immediately implies \cref{e:cm}, using Webster's formula \cite[(3.8)]{webster1978pseudo}.

In order to remove the assumption that $\vr_{j\kba}$ is invertible, we use an idea taken from \cite{li--luk} as before. We denote by $\phi^{\kba j}$ the adjugate matrix of $\vr_{j\kba}$, then as in \cite{li--son} we have,
\begin{equation} 
	\xi^{k} = \frac{\phi^{\kba j}\vr_{j}}{J(\vr)},
	\quad
	|\xi|^2 = \frac{\det (\vr_{j\kba})}{J(\vr)}.
\end{equation} 
Thus, the right-hand sides of \cref{e:pc,e:cm} are rational expressions in terms of derivatives of $\vr$ with denominators are some powers of $J(\vr)$. We replace $\vr$ by $\widetilde{\vr}: = \vr + C\vr^2$ for some constant $C>0$. By a direct calculation, $\det(\widetilde{\vr}_{j\kba}) = \det(\vr_{j\kba}) + 2CJ(\vr)$ on $M$. Therefore, $\widetilde{\vr}_{j\kba}$ is invertible on $M$ for every $C>0$ small enough since $J(\vr) \ne 0$. Observe that $\theta = -i\partial\vr = -i \partial \widetilde{\vr}$ on $M$. Therefore, the right-hand sides of \cref{e:pc} and \cref{e:cm} do not change when $C > 0$ varies. Passing through the limit when $C \to 0$, which is allowed since $J(\vr) \ne 0$ and $J(\widetilde{\vr}) \ne 0$ for all $C>0$ small enough, we conclude the proof.
\end{proof}

It was proved by Fefferman \cite{fefferman1976} that for each nondegenerate real hypersurface $M$, there exists a defining function $\vr_0$ such that $J(\vr_0) = 1 + O(\vr_0^{n+2})$. 
In this case, the formula for the Chern--Moser--Weyl tensor is greatly simplified. Indeed, we have the following
\begin{cor}\label{cor:chern--moser}
	Let $M$ be a nondegenerate real hypersurface defined by $\vr = 0$. Assume that $\vr_{w} \ne 0 $ and $J(\vr) = 1 + O(\vr^3)$, then the Chern--Moser--Weyl tensor of $(M,\theta: = -i\partial \vr)$ is given in the frame $Z_{\alpha}: = \partial_{\alpha} - (\vr_{\alpha}/\vr_w)\partial_w$ by
	\begin{align}\label{e:cmf}
	S_{\alpha\bba\gamma\sba} & =
	-\mathcal{R}^{\vr}_{\alpha\bba\gamma\sba}(\vr) + h^{j\kba} D^{\vr}_{\alpha\gamma}(\vr_{\kba}) D^{\vr}_{\bba\sba}(\vr_{j}) \notag \\
	& \quad +
	h_{\bba\sba}\xi^{\kba} D^{\vr}_{\alpha\gamma}(\vr_{\kba}) + h_{\alpha\gamma} \xi^{j} D^{\vr}_{\bba\sba}(\vr_{j}) - |\xi|^2 h_{\alpha\gamma} h_{\bba\sba}.
	\end{align}
\end{cor}
\begin{proof}
		If $J(\vr) = 1 + O(\vr^3)$ then clearly $D^{\vr}_{\alpha\bba} \log J(\vr) = D^{\vr}_{\alpha\beta} \log J(\vr) = 0$ on $M$ and the conclusion follows immediately from \cref{thm:chern--moser-general}.
\end{proof}
\begin{proof}[Proof of \cref{thm:chern-moser}]
	Since $\vr_{j\kba} = \delta_{jk}$ we immediately obtain that all the terms involving 3rd order derivatives in \eqref{e:pc} vanish, which proves \eqref{e:phc}, since $|\xi|^2 = |\partial \vr|^{-2}$. To show \eqref{e:chern-moser}, we use \eqref{e:phc} and compute,
	\begin{equation}
	R_{\gamma \bar \sigma} = h^{\alpha \bar \beta} R_{\alpha \bar \beta \gamma \bar \sigma} = \frac{1}{|\partial \vr|^2}\left((n+1) h_{\gamma \bar \sigma} - h_{\gamma \epsilon} h^{\epsilon \bar \delta} \right).
	\end{equation}
Together with \eqref{e:lilukRicci} we obtain
\begin{equation}
D^{\vr}_{\alpha \bar \beta} \log J(\vr) = \frac{h_{\alpha \epsilon}h^{\epsilon}_{\bar \beta}}{|\partial \vr|^2},
\end{equation}
which proves \eqref{e:chern-moser}. 
\end{proof}
Let us conclude this section by discussing the relation between our formula \cref{e:chern-moser} and the Chern--Moser normal form in \cite{chern1974real}. It is well-known (Eq. (6.20) in \cite{chern1974real}) that if a defining function is given in the normal form, then the Chern--Moser--Weyl tensor at the centered point can be identified with the coefficient of the 4th order term. This relation was further elaborated in \cite{huang2009monotonicity}, see also \cite{huang2017chern}. First, we suppose that the defining function takes the following form:
\begin{equation}\label{e:partialnormalform}
	\vr= \Im(w) - \|z\|^2+ F_4(z,\zba) + R(z,\zba, \Re(w)),
\end{equation}
where the fourth order term is 
\begin{equation} 
	F_4(z,\zba) = \frac{1}{4}\sum c_{\alpha\bar
		{\beta}\gamma \bar{\delta}}z_{\alpha} {\bar z_\beta}z_{\gamma}{\bar
		z_\delta},
\end{equation} 
and the remainder term $R(z,\zba, \Re(w))$ has ``weight'' and degree at least five; see \cite{chern1974real} for details. Since $F_4$ is real-valued, the coefficients in $F_4$ can be arranged to satisfy 
\begin{equation} 
	c_{\alpha\bar {\beta}\gamma \bar{\delta}}
	=
	c_{\gamma \bar {\beta}\alpha\bar{\delta}}
	=
	c_{\gamma\bar{\delta}\alpha\bar {\beta}},\ \quad
	\overline {c_{\alpha\bar {\beta}\gamma\bar{\delta}}}
	=
	c_{\beta\bar{\alpha}\delta\bar{\gamma}}.
\end{equation} 
Since $\rho_{\alpha}(0) = 0$ for $\alpha = 1,2,\dots , n$, it is readily seen that
\begin{equation} 
	\mathcal{R}^{\varrho}_{\alpha\bba\gamma\bar{\delta}}(\vr) \bigl|_0\, = \vr_{\alpha\bba\gamma\bar{\delta}}(0) = c_{\alpha\bar {\beta}\gamma \bar{\delta}}.
\end{equation} 
Moreover, the terms involving derivatives of order $\leq 3$ are
\begin{align}
	h_{\alpha\bba}\bigl|_0 = \delta_{\alpha\bba}, \qquad 
	h_{\alpha\beta}\bigl|_0 = 0, \qquad 
	D^{\vr}_{\alpha\gamma}(\vr_{\kba})\bigl|_0 = 0.
\end{align} 
Futhermore, since $\det [\vr_{j\kba}] = 0$ at the origin, we obtain that $|\xi|^2 = 0$ at the origin. Thus, the full curvature tensor and the torsion tensor at the origin are
\begin{align} 
	R_{\alpha\bar {\beta}\gamma \bar{\delta}}\, \bigl|_0\, 
	& =
	R\left(\partial_{\alpha}|_0, \partial_{\bba}|_0, \partial_{\gamma}|_0, \partial_{\bar{\delta}}|_0\right)\,
	=
	c_{\alpha\bar {\beta}\gamma \bar{\delta}},\\
	A_{\alpha\beta}|_0 & = \langle \tau(\partial_{\alpha}|_0) , \partial_{\beta}\,\bigl|_0\rangle = 0.
\end{align} 
If the defining function is normalized such that the coefficients $c_{\alpha\bar {\beta}\gamma \bar{\delta}}$ are completely tracefree (which is the case for the Chern--Moser normal form), then 
\begin{equation} 
	R_{\alpha\bba}\, \bigl|_0 = \sum_{\gamma} c_{\alpha\bar {\beta}\gamma \bar{\gamma}} = 0,
	\quad 
	R|_0 = 0.
\end{equation} 
Consequently, in this case
\begin{equation} 
	S_{\alpha\bar {\beta}\gamma \bar{\delta}}\, \bigl|_0\,
	=
	R_{\alpha\bar {\beta}\gamma \bar{\delta}}\, \bigl|_0\, 
	= c_{\alpha\bar {\beta}\gamma \bar{\delta}}.
\end{equation}
This is Eq. (6.20) in \cite{chern1974real} modulo a sign convention.

\section{The CR invariant one-form $X_{\alpha}$ on the real ellipsoids of revolution}\label{sec:Xalpha}
Let $M\subset \mathbb{C}^{n+1}$ be a strictly pseudoconvex CR manifold defined by $\varrho = 0$. There exists a unique pseudohermitian structure $\theta$ on $M$ which is volume-normalized with respect to $\zeta:=\iota^{\ast}\left(dz_1\wedge dz_2\wedge \cdots \wedge dz_{n+1}\right)$. Indeed, it follows from \cite{farris1986intrinsic} that if $J(\vr) = 1$ on $M$, then $\iota^{\ast}(i\bar{\partial}\vr)$ is the volume-normalized structure with respect to $\zeta$. In general, for an arbitrary defining function $\vr$, if we define $\vr_1= J(\vr)^{-1/(n+2)}\vr$, then $\vr_1$ is a defining function for $M$ which satisfies $J(\vr_1) = 1$ on $M$. Thus, $\widetilde{\theta} := J(\vr)^{-1/(n+2)}i\bar{\partial}\vr$
does not depend on the choice of the defining function and is volume-normalized. Consequently, there is a universal partial differential operator $\mathcal{P}$ such that the CR invariant one-form $X_{\alpha}$ defined by \cref{e:Xalphadef} is represented in the volume-normalized scale $\widetilde{\theta}$ by $\mathcal{P}(\vr)\bigl|_M$ for an arbitrary defining function~$\vr$. \cref{thm:xalpha}, which follows from the proposition below, implies that $\mathcal{P}$ is nontrivial and hence the nonvanishing of $X_{\alpha}$ is a ``generic'' phenomenon. 

In the rest of this section, we prove the following

\begin{prop}\label{p:nonTrivialX}
Let $E(a)$ be given as in \cref{thm:xalpha} and $\widetilde{\theta}$ be the unique pseudohermitian structure on $E(a)$ that is volume-normalized with respect to the section $\zeta: = dz_1 \wedge dz_2 \wedge dw\bigl|_{E(a)}$. Then,
	\begin{align}\label{e:xalpha1}
	X_{\alpha}(\widetilde{\theta})
	 =
	\frac{a^4\|z\|^6(|\partial\vr|^2\vr_w + 9a\|z\|^2\vr_{\wba})}{96|\partial\vr|^{17}}
	 (a\|z\|^2\vr_w + 4|\partial\vr|^2\vr_{\wba})\zba_{\alpha}dz_{\alpha} + a\|z\|^2(\vr^2_w - 4|\partial\vr|^2)dw.
	\end{align}
	Hence, $X_{\alpha}(\widetilde{\theta})$ is not identically zero on $E(a)$ for $a\ne 0$.
\end{prop}
\begin{proof}
	The proof is a matter of calculations, using the procedure described in the previous section and in particular \cref{thm:chern-moser}. To simplify notations, put $q(w) = |w|^2 + \Re (aw^2) - 1$ so that $\varrho = \|z\|^2 + q(w)$, with $\|z\|^2 = |z_1|^2 + \cdots + |z_n|^2$. On $E(a)$, $q(w) = -\|z\|^2$. Since $\vr_{j\kba} = \delta_{jk}$, we have $|\partial \varrho|^2 := |\partial \varrho|^{2}_{\omega} = \|z\|^2 + |\varrho_w|^2$. As proved in \cite{son2019semi}, $|H|^2 = |\partial\vr|^{-2}$ is the transverse curvature of $\vr$. We have by \eqref{e:levimt}, \eqref{e:dab} and \eqref{e:inverseLevi}
	\begin{equation}
		h_{\alpha\bba} = \delta_{\alpha \beta} + \frac{\zba_\alpha z_\beta}{|\vr_w|^2}, \qquad h_{\alpha\beta}
		= D_{\alpha\beta}^{\vr}(\vr)
		=
		\frac{a\zba_{\alpha}\zba_{\beta}}{\varrho_w^2}, \qquad h^{\alpha \bba} = \delta_{\alpha \beta} - \frac{z_\alpha \zba_\beta}{|\partial \vr|^2},
	\end{equation}
such that 
\begin{equation}
h^{\mu \nu} = h_{\bba}^{\mu} h^{\nu \bba} = h_{\bba \sba} h^{\mu \sba} h^{\nu \bba} = \frac{a z_\nu z_\mu \vr_w^2}{|\partial \vr|^4}.
\end{equation}
	Then, \cref{thm:chern-moser} gives
	\begin{align}
		S_{\beta\aba\rho\sba}
		& = 
		- \frac{a^2\zba_{\beta} z_{\alpha} \zba_{\rho} z_{\sigma}}{|\partial\varrho|^2|\varrho_w|^4} 
		- \frac{a^2\|z\|^4\left(h_{\beta\aba}h_{\rho\sba} + h_{\rho\aba}h_{\beta\sba}\right)}{(n+1)(n+2)|\partial\varrho|^6} \notag \\
		& \qquad + \frac{a^2\|z\|^2\left(h_{\beta\aba} \zba_{\rho}z_{\sigma} + h_{\rho\aba}\zba_{\beta} z_{\sigma} + h_{\beta\sba} \zba_{\rho} z_{\alpha} + h_{\rho\sba} \zba_{\beta}z_{\alpha} \right) }{(n+2)|\varrho_w|^2 |\partial \varrho|^4}.
	\end{align}
	Raising the indices, using,
	\begin{equation}
	h^{\alpha \bba} z_{\beta} = \frac{z_{\alpha} |\vr_w|^2}{|\partial \vr|^2},
	\end{equation}
	we obtain,
	\begin{align}
		S_{\alpha}{}^{\mu\bar{\nu}\gamma}
		& =
		-\frac{a^2|\vr_w|^2 \zba_{\alpha}z_{\mu} \zba_{\nu} z_{\gamma}}{|\partial \vr|^8}
		-
		\frac{a^2\|z\|^4\left(h^{\gamma\bar{\nu}}\delta_{\alpha\mu} + h^{\mu\bar{\nu}}\delta_{\alpha\gamma}\right)}{(n+1)(n+2)|\partial\vr|^6} \notag \\
		& \quad +
		\frac{a^2\|z\|^2\left(\delta_{\alpha\mu}\zba_{\nu}z_{\gamma}|\vr_w|^2/|\partial\vr|^2 + h^{\mu\bar{\nu}}\zba_{\alpha} z_{\gamma} + \delta_{\alpha\gamma}z_{\mu}\zba_{\nu}|\vr_w|^2/|\partial\vr|^2 + h^{\gamma\bar{\nu}}z_{\mu}\zba_{\alpha} \right)}{(n+2)|\partial\vr|^6},
	\end{align}
	and
	\begin{align}
		S^{\beta\aba\rho\sba}
		& =
		- \frac{a^2\|z\|^4\left(h^{\beta\aba}h^{\rho\sba} + h^{\rho\aba}h^{\beta\sba}\right)}{(n+1)(n+2)|\partial\varrho|^6} - \frac{a^2 |\varrho_w|^4 z_{\beta} \zba_{\alpha} z_{\rho} \zba_{\sigma}}{|\partial\varrho|^{10}} \notag \\
		& \qquad +
		\frac{a^2\|z\|^2|\varrho_w|^2\left(h^{\beta\aba} z_{\rho}\zba_{\sigma} + h^{\rho\aba} z_{\beta} \zba_{\sigma} + h^{\beta\sba} z_{\rho} \zba_{\alpha} + h^{\rho\sba} z_{\beta} \zba_{\alpha} \right) }{(n+2) |\partial\varrho|^8}.
	\end{align}
	From this, and since, 
\begin{equation}
h_{\alpha \bba} z_\alpha \zba_{\beta} = \frac{\|z\|^2 |\partial \vr|^2}{|\vr_w|^2},
\end{equation}	
		we can calculate the norm of the Chern--Moser--Weyl tensor:
	\begin{align}
		|S|^2 
		=
		S^{\beta\aba\rho\sba} S_{\beta\aba\rho\sba}
		= \frac{n(n-1)}{(n+1)(n+2)}\frac{a^4\|z\|^8}{|\partial \varrho|^{12}}.
	\end{align}
	We point out that we have used the completely tracefree property of the Chern--Moser--Weyl tensor to simplify our computations.
	
	To determine the unique volume-normalized pseudohermitian structure on $E(a)$ with respect to $\zeta: = \iota^{\ast}dz$, observe that the Levi-Fefferman determinant satisfies (see \cite[Lemma 2.2]{li--luk} and its proof),
	\begin{equation}
		J(\varrho) = \det \left[\varrho_{j\kba}\right]( -\varrho + |\partial\varrho|^2),
	\end{equation}
	such that $J(\varrho) = |\partial\varrho|^2$ on $E(a)$. If we put
	\begin{equation}
		u = -\frac{1}{n+2}\log J(\varrho) = -\frac{1}{n+2} \log |\partial\varrho|^2,
	\end{equation}
	and 
	\begin{equation}
		\widetilde{\theta} = e^u \theta,
	\end{equation}
	then by \cite{farris1986intrinsic}, $\widetilde{\theta}$ is volume-normalized with respect to $\zeta$ and hence is pseudo-Einstein by \cite{lee1988pseudo}. 
	
	The pseudohermitian invariants of $\widetilde{\theta}$ will be indicated with a tilde.	It is well-known that the Chern--Moser--Weyl tensor transforms as follows (\cite{webster1978pseudo}): 
	\begin{equation}
		\widetilde{S}_{\beta\aba\rho\sba} = e^u {S}_{\beta\aba\rho\sba}, 
		\quad 
		\widetilde{S}_{\alpha}{}^{\mu\bar{\nu}\gamma}
		= e^{-2u} S_{\alpha}{}^{\mu\bar{\nu}\gamma},
	\end{equation}
	and
	\begin{equation}\label{e:cmnorm}
		|\widetilde{S}|^2_{\widetilde{\theta}} = e^{-2u} |S|^2_{\theta}
		=
		\frac{n(n-1)a^4}{(n+1)(n+2)}\|z\|^8|\partial \varrho|^{4/(n+2) - 12}.
	\end{equation}
	For $\widetilde{\theta}$, we shall use the same holomorphic frame $Z_{\alpha}$ so that the Levi matrix becomes
	\begin{equation} 
		\widetilde{h}_{\alpha\bba}
		=
		-id\widetilde{\theta}(Z_{\alpha},Z_{\bba}) = e^u h_{\alpha\bba}.
	\end{equation} 
	Differentiating \cref{e:cmnorm} with respect to $Z_{\alpha}$ we have
	\begin{equation} \label{e:NablaS}
		\nabla_{\alpha} |\widetilde S|^2
		=
		\frac{n(n-1)a^4}{(n+1)(n+2)}\frac{\|z\|^6 \zba_{\alpha}}{|\partial\vr|^{14-4/(n+2)}} \left(4|\partial\vr|^2 + a(6-2/(n+2)) \|z\|^2\vr_{\wba}/\vr_w\right).
	\end{equation} 
	To simplify computations, we can use the fact that, on $E(a)$, $\|z\|^2 = -q$ and $|\partial\vr|^{2} = -q + |q_w|^2$. Since $q$ only involves the variable $w$, the differentiation reduces essentially to~$\partial_w$.
	
	To compute the term in $X_{\alpha}$ which involves the torsion, we use the following formula (Eq. (2.16) in \cite{lee1988pseudo})
	\begin{equation}
		-i\widetilde{A}_{\alpha\beta}
		=
		-i{A}_{\alpha\beta} + u_{\alpha,\beta} - u_\alpha u_\beta,
	\end{equation}
	where $A_{\alpha\beta}$ is computed in \cref{cor:torsion}, which in this case is given by
	\begin{equation}
	-iA_{\alpha\beta}
	=
	\frac{a\zba_{\alpha}\zba_{\beta}}{\varrho_w^2 |\partial\varrho|^2}.
	\end{equation}
By a direct calculation, we have
	\begin{equation}
		u_{\alpha} = Z_{\alpha} u
		=
		\frac{a\varrho_{\wba} \zba_{\alpha}}{(n+2)\varrho_w|\partial \varrho|^2}.
	\end{equation}
	Furthermore, by using the formula for the connection forms \cref{e:cf}, we have
	\begin{equation}
		\omega_{\alpha}{}^{\gamma}(Z_{\beta})
		=
		\frac{a\zba_{\alpha}\zba_{\beta}z_{\gamma}}{\varrho_w^2 |\partial\varrho|^2},
	\end{equation}
	and hence
	\begin{align}
		u_{\alpha,\beta}
		& =
		Z_{\beta}u_{\alpha} - \omega_{\alpha}{}^{\gamma}(Z_{\beta}) u_{\gamma} \notag \\
		& = \frac{a\zba_{\alpha}\zba_{\beta}}{(n+2)\varrho_{w}^2|\partial\varrho|^2}\left(\frac{a\varrho_{\wba}^2}{|\partial\varrho|^2} + \frac{a\varrho_{\wba}}{\varrho_w} - 1\right) - \frac{a^2\zba_{\alpha}\zba_{\beta} \|z\|^2 \varrho_{\wba}}{(n+2)\varrho_{w}^3 |\partial\varrho|^4} \notag \\
		& = 
		\frac{a\zba_{\alpha}\zba_{\beta}}{(n+2)\varrho_{w}^2|\partial\varrho|^2}\left(\frac{2a\varrho_{\wba}^2}{|\partial\varrho|^2} - 1\right).
	\end{align}
	Therefore,
	\begin{align}
		-i\widetilde{A}_{\alpha\beta}
		& =
		-i{A}_{\alpha\beta} + u_{\alpha,\beta} - u_\alpha u_\beta \notag \\
		& =
		\frac{1}{n+2}\frac{a\zba_{\alpha}\zba_{\beta}}{\varrho_w^2 |\partial\varrho|^2}\left(n+1+ \frac{a(2n+3)\varrho_{\wba}^2}{(n+2)|\partial\varrho|^2}\right).
	\end{align}
	Differentiating this with respect to $Z_{\bar{\gamma}}$, we have that
	\begin{align}
		-i Z_{\bar{\gamma}}\widetilde{A}_{\alpha\beta}
		=
		\frac{a}{(n+2)\varrho_w^2 |\partial\varrho|^2}\left(Q(\delta_{\alpha\gamma}\zba_{\beta} + \delta_{\beta\gamma}\zba_{\alpha}) + \zba_{\alpha}\zba_{\beta}z_{\gamma}P\right),
	\end{align}
	where
	\begin{equation}
		Q: = n+1+ \frac{a(2n+3)\varrho_{\wba}^2}{(n+2)|\partial\varrho|^2},
	\end{equation}
	and
	\begin{align}
		P 
		& = \left(\frac{2}{|\varrho_w|^2} + \frac{a\varrho_w}{|\partial\varrho|^2 \varrho_{\wba}}\right)Q +\frac{(2n+3)a^2}{n+2} \left(\frac{|\varrho_w|^2}{|\partial\varrho|^4} - \frac{2}{|\partial\varrho|^2} \right) \notag \\
		& = 
		\frac{2(n+1)}{|\varrho_w|^2} + \frac{a(n+1)\varrho_w}{|\partial\varrho|^2 \varrho_{\wba}}
		+ \frac{2a(2n+3)\varrho_{\wba}}{(n+2)\varrho_w |\partial\varrho|^2 } 
		- \frac{2a^2(2n+3) \|z\|^2}{(n+2)|\partial\varrho|^4}.
	\end{align}
	The connection forms with respect to $\widetilde{\theta}$ are given as follows (see \cite{dragomir--tomassini}, page 137, equation (2.41)):
	\begin{equation}
		\widetilde{\omega}_{\alpha}{}^{\mu}(Z_{\gba})
		=
		{\omega}_{\alpha}{}^{\mu}(Z_{\gba})
		-
		u^{\mu} h_{\alpha\gba}
		=
		\left(\xi^{\mu} - u^{\mu}\right)h_{\alpha\gba}.
	\end{equation}
	Thus, 
	\begin{align}
		\widetilde{A}_{\alpha\beta,\gba}
		& =
		Z_{\gba}\widetilde{A}_{\alpha\beta} - \widetilde{\omega}_{\alpha}{}^{\mu}(Z_{\gba}) \widetilde{A}_{\mu\beta} - \widetilde{\omega}_{\beta}{}^{\mu}(Z_{\gba}) \widetilde{A}_{\alpha\mu} \notag \\
		& = Z_{\gba}\widetilde{A}_{\alpha\beta} - \phi_{\beta}h_{\alpha\gba} - \phi_{\alpha}h_{\beta\gba}.
	\end{align}
	Here $\phi_{\beta} : = \sum_{\mu = 1}^{n} \widetilde{A}_{\mu\beta} (\xi^{\mu} - u^\mu)$. Note that $\phi_{\beta}$ is not tensorial. 
	
	Recall from \cite{gover2005sublaplace} and section 2.3 of \cite{case2017p} that the CR analogue of the Cotton tensor $V_{\alpha\bba\gamma}$ is defined by
	\begin{equation}\label{e:VTensor}
	V_{\alpha\bba\gamma}
		= 
		A_{\alpha\gamma,\bba}
		+ i P_{\alpha\bba,\gamma} - iT_\gamma h_{\alpha\bba} -2i T_{\alpha}h_{\gamma\bba},
	\end{equation}
	which satisfies \cite[Lemma~2.1]{case2017p}
	\begin{equation}\label{e:sv}
		S_{\alpha\bba\gamma\sba,}{}^{\sba} = -i n V_{\alpha\bba\gamma}.
	\end{equation} 
	Here the pseudohermitian Schouten tensor $P_{\alpha \bar \beta}$ is given by
	\begin{equation}\label{e:Schouten}
	P_{\alpha \bar \beta} = \frac{1}{n+2} \left(R_{\alpha \bar \beta} - \frac{R h_{\alpha \bar \beta}}{2 (n+1)} \right)
	\end{equation}
	and 
	\begin{equation}\label{e:TTensor}
		T_{\alpha} = \frac{1}{n+2} \left(\frac{R_{,\alpha}}{2(n+1)} -i A_{\alpha\sigma,}{}^{\sigma}\right).
	\end{equation}
These expressions are simplified on pseudo-Einstein manifolds as follows:
	\begin{lem} \label{lem:VTensor}
	Let $(M,\eta)$ be a pseudo-Einstein CR manifold of dimension $2n+1 \geq 5$. Then
		\begin{equation} 
			V_{\alpha\bba\gamma}
			=
			A_{\alpha\gamma,\bba} + \frac{i(R_{,\gamma}h_{\alpha\bba} + R_{,\alpha} h_{\gamma\bba})}{n(n+1)},
		\end{equation} 
		\begin{equation} \label{e:cmv}
			S_{\rho}{}^{\alpha\bba\gamma} V_{\alpha\bba\gamma} = S_{\rho}{}^{\alpha\bba\gamma} A_{\alpha\gamma,\bba}.
		\end{equation} 
	\end{lem}
	\begin{proof}
	If $\theta$ is pseudo-Einstein, then $R_{\alpha\bba} = (R/n) h_{\alpha\bba}$ and thus from \eqref{e:Schouten} it follows that
	\begin{equation} 
		P_{\alpha\bba,\gamma} 
		=
		\frac{1}{2n(n+1)} R_{,\gamma} h_{\alpha\bba}.
	\end{equation} 
	Using the identity $R_{,\alpha} - i(n-1)A_{\alpha\sigma,}{}^{\sigma} = R_{\alpha\bba,}{}^{\bba}$ (Eq. (2.11) in \cite{lee1988pseudo}), we find that \eqref{e:TTensor} becomes
	\begin{equation} 
		T_{\alpha}
		= -\frac{1}{2n(n+1)}R_{,\alpha}.
	\end{equation} 
	Plugging these expressions into \eqref{e:VTensor} we obtain that
	\begin{equation*} 
		V_{\alpha\bba\gamma}
		=
		A_{\alpha\gamma,\bba} + \frac{i(R_{,\gamma}h_{\alpha\bba} + R_{,\alpha} h_{\gamma\bba})}{n(n+1)}.
	\end{equation*} 
	Then, \cref{e:cmv} follows since the Chern--Moser--Weyl tensor is completely tracefree. The proof is complete.
	\end{proof}
To compute further in the proof of \cref{p:nonTrivialX}, we observe that
\begin{align} \label{e:identity1}
	\zba_{\alpha}z_{\mu} \zba_{\nu} z_{\gamma}\left(\zba_{\mu} \delta_{\gamma\nu} + \zba_{\gamma}\delta_{\mu\nu}\right) 
	= 2\|z\|^4 \zba_{\alpha},\\
	\left(h^{\gamma\bar{\nu}}\delta_{\alpha\mu} + h^{\mu\bar{\nu}}\delta_{\alpha\gamma}\right) \left(\zba_{\mu} \delta_{\gamma\nu} + \zba_{\gamma}\delta_{\mu\nu}\right)
	=
	2\zba_{\alpha} \left(n-1+\frac{2|\vr_w|^2}{|\partial\vr|^2}\right) \\
	\left(\frac{\delta_{\alpha\mu}\zba_{\nu}z_{\gamma}|\vr_w|^2}{|\partial\vr|^2} + h^{\mu\bar{\nu}}\zba_{\alpha} z_{\gamma} + \frac{\delta_{\alpha\gamma}z_{\mu}\zba_{\nu}|\vr_w|^2}{|\partial\vr|^2} + h^{\gamma\bar{\nu}}z_{\mu}\zba_{\alpha} \right)\left(\zba_{\mu} \delta_{\gamma\nu} + \zba_{\gamma}\delta_{\mu\nu}\right)\notag \\
	=
	2\left((n-1)|\partial\vr|^2 +4|\vr_w|^2\right)\frac{\|z\|^2\zba_{\alpha}}{|\partial\vr |^2}. \label{e:identity3}
\end{align}
	Since $S_{\rho}{}^{\alpha\bar{\gamma}\beta} h_{\alpha\gba} = S_{\rho}{}^{\alpha\bar{\gamma}\beta} h_{\beta\gba} = 0$, we have by \cref{lem:VTensor} and using the identities \eqref{e:identity1}--\eqref{e:identity3}, that,
	\begin{align}\label{e:SV}
		-i\widetilde{S}_{\alpha}{}^{\mu\bar{\nu}\gamma} \widetilde{V}_{\mu\bar{\nu}\gamma}
		& = 
		-ie^{-2u}S_{\alpha}{}^{\mu\bar{\nu}\gamma} \widetilde{A}_{\mu\gamma,\bar{\nu}} \notag \\
		& = 
		\frac{n(n-1)}{(n+1)(n+2)^2}\frac{a^3 \|z\|^6 \zba_{\alpha}}{|\partial\vr|^{10-4/(n+2)}\vr_w^2}\left(2Q - |\vr_w|^2 P\right) \notag \\
		& = 
		\frac{n(n-1)}{(n+1)(n+2)^2}\frac{a^4 \|z\|^6 \zba_{\alpha}}{|\partial\vr|^{14 - 4/(n+2)}}
		\left(\frac{2(2n+3) a\|z\|^2 \varrho_{\wba}}{(n+2)\vr_w} - (n+1)|\partial\varrho|^2\right).
	\end{align}
	Therefore, putting \cref{e:NablaS,e:sv,e:SV} together and setting $n=2$, we obtain:
	\begin{equation}\label{e440}
		\widetilde{X}_{\alpha}
		= 
		-i \widetilde S_{\alpha}{}^{\mu\nba\gamma} \widetilde V_{\mu\nba\gamma} + \frac{1}{4} \nabla_{\alpha} |\widetilde S|^2
		=
		\frac{1}{24}\frac{a^4\|z\|^6 \zba_{\alpha}}{|\partial\varrho|^{13}}\left(|\partial\varrho|^2 +9 a\|z\|^2\vr_{\wba}/\vr_w\right).
	\end{equation}
Thus, $\widetilde{X}_{\alpha} \not\equiv 0$ on $E(a)$, as desired. Plugging in the coframe $\widetilde{\theta}^{\alpha} = dz_{\alpha} + (u^{\alpha} - \xi^{\alpha})\partial\vr$ and simplifying the result, we obtain \cref{e:xalpha1}. The proof is complete.
\end{proof}
\begin{rem} Using \cref{e440}, we can further show that the divergence $\widetilde{\nabla}^{\alpha} \widetilde{X}_{\alpha}$ is nontrivial by direct calculations. Indeed, computing at $w\ne 0$, using
\[
	\widetilde{\nabla}^{\alpha} \widetilde{X}_{\alpha}
	=
	e^{-u}\left(Z_{\sba} (\widetilde{X}_{\alpha})h^{\sba\alpha} - 2 (\xi^{\mu} - u^{\mu}) \widetilde{X}_{\mu}\right),
\]
and letting $w \to 0$, we find that 
\[
	\widetilde{\nabla}^{\alpha} \widetilde{X}_{\alpha}\bigl|_{w = 0} = -\tfrac{1}{24}a^4(9 a^2+1) \ne 0
\]
if $a\ne 0$. Thus, as briefly sketched in \cref{sec:intro}, $E(a)$ furnishes a counterexample for the Hirachi conjecture in the nonspherical case.
\end{rem}
\section{A hypersurface with parallel Chern--Moser--Weyl tensor}

If $M^5$ is CR spherical, then both $X_{\alpha}$ and $\mathcal{I}'$ are trivial. More generally, if $M^5$ admits a contact form such that
\begin{equation}\label{e:parallelcm}
	S_{\alpha\bba\gamma\sba,\rho} = S_{\alpha\bba\gamma\sba,\bar{\rho}} = 0,
\end{equation} 
then $X_{\alpha} = 0$ and hence $\mathcal{I}'$ is CR invariant in this case. Using \cref{thm:chern-moser} and \cref{cor:torsion}, 
we give an explicit example of a nonspherical CR manifold such that the conditions in \cref{e:parallelcm} hold.
\begin{example}\label{ex:parallelchernmoser}
Consider the ellipsoidal tube $E=E(1,1,\dots,1)$ given by $\vr = 0$, where
\begin{equation*}
	\vr : = \sum_{j=1}^{n+1}|z_j|^2 + \Re \sum_{j=1}^{n+1} z_j^2 -1.
\end{equation*}
We point out that this real hypersurface has been studied in various papers, e.g., \cite{marugame2016renormalized,ebenfelt2018family}.
By direct calculations,
\begin{equation*}
	|\partial \vr|^2 = \sum_{j=1}^{n+1} |z_j + \zba_j|^2 = 2\vr +2, \quad 	h_{\alpha\beta} = h_{\alpha\bba} = \delta_{\alpha\beta} + \frac{\vr_{\alpha}\vr_{\beta}}{\vr_w^2}, \quad w: = z_{n+1}.
\end{equation*}
With $\theta: = i\bar{\partial}\vr$, the Tanaka--Webster connection forms are
\begin{equation}\label{e:cfexample}
	\omega_{\beta}{}^{\gamma}
	=
	\tfrac12 \left(\vr_{\gba} h_{\beta\mu} \theta^{\mu} + \vr_{\gba}h_{\beta\bar{\mu}} \theta^{\bar{\mu}} - i \delta_{\beta\gamma}\, \theta\right),
\end{equation}
and the pseudohermitian curvature tensor is
\begin{equation}\label{e:curexample}
	R_{\alpha\bba\gamma\sba}
	=
	- \frac{h_{\alpha\gamma}h_{\bba \sba}}{2} + \frac{h_{\alpha\bba}h_{\gamma\sba} + h_{\alpha\sba}h_{\gamma\bba}}{2}.
\end{equation}
Taking the trace, we see that $R_{\gamma\sba} = \frac{n}{2} h_{\gamma\sba}$ and $R = \frac{n^2}{2}$. Hence, $\theta$ is pseudo-Einstein. On the other hand, one can derive from \cref{e:cfexample,e:curexample}, that 
\begin{equation*} 
	R_{\alpha\bba\gamma\sba, \epsilon}
	=
	R_{\alpha\bba\gamma\sba, \bar \epsilon}
	=
	R_{\alpha\bba\gamma\sba, 0} = 0.
\end{equation*} 
In other words, $\theta$ is symmetric and, in particular, the Chern--Moser--Weyl tensor is parallel. This implies that $X_\alpha = 0$.

By removing the trace from $R_{\alpha\bba\gamma\sba}$, we obtain
\begin{equation}\label{e:cm1}
	S_{\alpha\bba\gamma\sba}
	=
	- \frac12 h_{\alpha\gamma}h_{\bba \sba} + \frac{h_{\alpha\bba}h_{\gamma\sba} + h_{\alpha\sba}h_{\gamma\bba}}{2(n+1)}
\end{equation}
Thus,
\begin{equation}\label{e:cmexample}
	|S|^2 = \frac{n(n-1)}{4(n+1)},
\end{equation}
which implies in particular that $E$ is nonumbilical if $n\geq 2$ (the case $n = 1$ was treated in \cite{ebenfelt2018family}). 
Thus, in dimension five, the invariant $\mathcal{I}^{\prime}(\theta)$ from \eqref{e:defIprime} is given by
\begin{equation*}
	\mathcal{I}^{\prime}(\theta) = \frac{1}{36}.
\end{equation*}
\end{example}
\begin{example}
Note that the real hypersurface $E$ in \cref{ex:parallelchernmoser} is noncompact. However, $E$ is locally CR equivalent to the compact Reinhardt hypersurface defined by $\Sigma_r: = \{\widetilde{\vr}_r = 0\}$, where
 \begin{equation} 
	 \widetilde{\vr}_r(z, \zba) = \sum_{j=1}^{n+1} \left(\log |z_j|\right)^2 - r^2, \quad r>0.
 \end{equation} 
 There is a unique pseudohermitian structure $\widetilde{\theta}_r$ on $\Sigma_r$ such that \cref{e:cmexample} holds on $\Sigma_r$. For this structure, the local considerations on $\Sigma_r$ and $E$ agree. Moreover, with this normalization, $\widetilde{\theta}_r$ is invariant under CR diffeomorphisms and its volume $\int_M \widetilde{\theta}_r \wedge (d\widetilde{\theta}_r)^n$ is a global invariant. Alternatively, for any $\theta$, the integral $\int_M |S_{\alpha\bba\gamma\sba}|^{n+1} \theta \wedge (d\theta)^n$ is invariant; the latter interpretation is valid for all cases (i.e., regardless whether $M$ is nowhere umbilical or not), albeit the integrand $|S_{\alpha\bba\gamma\sba}|^{n+1}$ is not polynomial in $S$ when $n+1$ is not even.
 
A seemingly more refined polynomial CR invariant that is built from $S$ is the following:
\[
	S^{n+1} 
	:=
	S_{\alpha_1}{}^{\alpha_2}{}_{\mu_1}{}^{\mu_2}S_{\alpha_2}{}^{\alpha_3}{}_{\mu_2}{}^{\mu_3}\, \cdots \, S_{\alpha_n}{}^{\alpha_{n+1}}{}_{\mu_n}{}^{\mu_{n+1}} S_{\alpha_{n+1}}{}^{\alpha_1}{}_{\mu_{n+1}}{}^{\mu_1} \quad (n+1\  \text{factors}).
\]
Clearly, the integral $\int_M S^{n+1} \theta \wedge (d\theta)^n$ is invariant.
On $\Sigma_r$, the Chern--Moser--Weyl tensor $S_{\alpha\bba\gamma\sba}$ is parallel and hence $S^{n+1}$ is constant on $\Sigma_r$. This constant can be computed explicitly from \cref{e:cm1}. Indeed, raising the indices on both sides of \cref{e:cm1} yields
\[
	S_{\alpha_1}{}^{\alpha_2}{}_{\mu_1}{}^{\mu_2}
	=
	c_1 h_{\alpha_1\mu_1} h^{\alpha_2\mu_2} + d_1 \left(\delta_{\alpha_1}^{\alpha_2} \delta_{\mu_1}^{\mu_2} + \delta_{\alpha_1}^{\mu_2} \delta_{\mu_1}^{\alpha_2} \right),
\]
with $c_1 = -1/2$ and $d_1 = 1/(2(n+1))$. For each $k\geqslant 1$,
\[
	S{[k]}:= 	S_{\alpha_1}{}^{\alpha_2}{}_{\mu_1}{}^{\mu_2}S_{\alpha_2}{}^{\alpha_3}{}_{\mu_2}{}^{\mu_3}\, \cdots \, S_{\alpha_k}{}^{\alpha_{k+1}}{}_{\mu_k}{}^{\mu_{k+1}}
\]
takes a similar form:
\[
	S{[k]}_{\alpha_1}{}^{\alpha_2}{}_{\mu_1}{}^{\mu_2}=c_k h_{\alpha_1\mu_1} h^{\alpha_2\mu_2} + d_k \left(\delta_{\alpha_1}^{\alpha_2} \delta_{\mu_1}^{\mu_2} + \delta_{\alpha_1}^{\mu_2} \delta_{\mu_1}^{\alpha_2} \right),
\]
with corresponding coefficients $c_k$ and $d_k$ satisfying the following recursion relations:
\[
	c_{k+1} = n c_1 c_k  +  2\, c_k d_1 + 2\, c_1 d_k,
	\quad 
	d_{k+1} = 2\, d_1 d_k.
\]
Solving for $d_k$ yields $d_k = \frac12 (n+1)^{-k}$. The recursion formula for the $c_k$'s becomes
\[
	c_{k+1} = - \frac{1}{2}\left(\frac{n^2 + n -2}{n+1} c_k + \frac{1}{(n+1)^k}\right).
\]
This relation can be solved by setting $y_k = (n+1)^{k-1} c_k + 1/(n^2+n)$. Omitting the detailed calculations, we present the result for $c_k$:
\[
c_k = \frac{1}{n(n+1)^k} \left[\left(\frac{2-n-n^2}{2}\right)^k -1\right].
\]
Consequently,
\begin{align}\label{e:last}
	S^{n+1}
	& = 
	S{[n]}_{\alpha_1}{}^{\alpha_2}{}_{\mu_1}{}^{\mu_2} \cdot  S_{\alpha_2}{}^{\alpha_1}{}_{\mu_2}{}^{\mu_1} \notag \\
	& = n^2 c_1 c_n + 2n c_1 d_n + 2n c_n d_1 + 4 d_1 d_n  \notag \\
	& = \left[\frac{1}{n+1} - \frac{n}{2}\right]^{n+1}.
\end{align}
Hence, $S^{n+1}$ is a real constant which is positive when $n$ is odd and negative when $n$ is even.

Finally, one can verify that $\vol(\Sigma_r, \widetilde{\theta}_r) = C r^{-n-1}$ for some constant $C$ depending on $n$. This and either \cref{e:cmexample} or \cref{e:last} imply that $\Sigma_r$'s are not globally equivalent for different values of~$r$. The last observation was proved in the cases $n=1$ and $n=2$ by Burns-Epstein \cite{burns1988global} and Marugame \cite{marugame2016renormalized}, respectively, using the same strategy. Precisely, in both cases, the authors computed the Burns--Epstein invariant for $\Sigma_r$'s which turn out to be different for different values of $r$. When $n\geq 3$, the Burns-Epstein is difficult to compute, however, as pointed out to the authors by the referee, Marugame's consideration can be generalized easily to treat the general case.
\end{example}
\section*{Acknowledgement}
The authors would like to thank the referee for suggesting the detailed argument that disproves the Hirachi conjecture and for many other useful comments.

\end{document}